\documentclass[12pt]{amsart}
\usepackage{a4}
\usepackage{amsbsy}
\newtheorem{theorem}{Theorem} 
\newtheorem{lemma}[theorem]{Lemma} 
\newtheorem{corollary}[theorem]{Corollary} 
\theoremstyle{definition}
\newtheorem{definition}[theorem]{Definition}
\newcommand{\ab}[1]{{\mathbf{#1}}} 

\DeclareMathAlphabet{\mathbfsl}{OT1}{cmr}{bx}{it}
\newcommand{\tupBold}[1]{\mathbfsl{#1}}
\newcommand{\N}{\mathbb{N}}
\newcommand{\Z}{\mathbb{Z}} 
\newcommand{\uli}[1]{\underline{#1}}

\newcommand{\Pol}{\mathrm{Pol}}

\newcommand{\Dep}{\mathrm{Dep}}
\newcommand{\adeg}{\mathrm{adeg}}

\newcommand{\algop}[2]{( {#1}, {#2} )}
\newcommand{\vb}[1]{\tupBold{#1}}
\newcommand{\vbr}[2]{\vb{#1}^{(#2)}}
\newcommand{\tup}[3]{(#1_{#2},\dots,#1_{#3})}
\newcommand{\wt}{\operatorname{wt}}
\newcommand{\NP}{\mathrm{NP}}
\newcommand{\PP}{\mathrm{P}}
\newcommand{\Pot}{\mathcal{P}}

\title[Systems of equations in supernilpotent algebras]{Solving systems of equations in supernilpotent algebras} %

\author[E.\ Aichinger]{Erhard Aichinger}

\address{Institute for Algebra, Johannes Kepler University Linz, Linz, Austria}

\email{erhard@algebra.uni-linz.ac.at}



\thanks{Supported by the Austrian Science Fund (FWF):P29931.}

\begin{document}
\bibliographystyle{alpha}

\maketitle

\begin{abstract}
  Recently, M.\ Kompatscher proved that for each finite supernilpotent algebra
  $\ab{A}$
  in a congruence modular variety, there is a polynomial time algorithm
  to solve polynomial equations over this algebra. Let $\mu$ be the maximal
  arity of the fundamental operations of $\ab{A}$, and let
  \[ d := |A|^{\log_2 (\mu) + \log_2 (|A|) + 1}.\]
   Applying a method that G.\ K{\'a}rolyi and C.\ Szab\'{o} had used to solve
  equations over finite nilpotent rings, we show that for $\ab{A}$,
  there is $c \in \N$ such that a solution of every system of $s$
  equations in $n$ variables
  can be found
  by testing 
  at most $c n^{sd}$ (instead of all $|A|^n$ possible)
  assignments to the variables. This also yields new information
  on some circuit satisfiability problems. 
\end{abstract}

\section{Introduction} \label{sec:intro}
We study systems of polynomial equations over
a finite algebraic structure $\ab{A}$. Such a system
is given by equations of the form $p(x_1,\ldots, x_n)
\approx q(x_1, \ldots, x_n)$, where $p,q$
are polynomial terms of $\ab{A}$; a polynomial term
of $\ab{A}$ is a term of the algebra $\ab{A}^*$ which
is obtained by expanding $\ab{A}$ with one nullary
function symbol for each $a \in A$.
A \emph{solution} to a system $p_i(x_1, \ldots, x_n)
\approx q_i (x_1, \ldots, x_n)$ ($i = 1, \ldots, s$)
is an element $\vb{a} = \tup{a}{1}{n} \in A^n$ such that
$p_i^{\ab{A}} (\vb{a}) =  q_i^{\ab{A}} (\vb{a})$ for all
$i \in \{1, \ldots, s\}$.
The problem to decide whether such a solution exists
has been called ${\textsc{PolSysSat} (\ab{A})}$, and ${\textsc{PolSat} (\ab{A})}$
if the system consists of one single equation, and
the terms of the input are encoded as strings
over $\{x_1, \ldots, x_n\} \cup A \cup F$, where
$F$ is the set of function symbols of $\ab{A}$.
A survey of results on the computational complexity
of this problem is given, e.g.,  in \cite{IK:SIMC,Ko:TESP}.
In algebras such as groups, rings or Boolean algebras,
one can reduce an equation $p (\vb{x}) \approx q (\vb{x})$ to an equation of the form $f (\vb{x}) \approx y$, where
$y \in A$. A system of equations of this form
then has the form $f_i (\vb{x}) \approx y_i$
($i = 1, \ldots, s$).
For $n \in \N$, let $\Pol_n (\ab{A})$ denote the $n$-ary polynomial
functions on $\ab{A}$ \cite[Definition~4.4]{MMT:ALVV}.
For a finite nilpotent ring or group $\ab{A}$,
\cite{Ho:TCOT} establishes the existence of
a natural number $d_{\ab{A}}$ such that
for every $f \in \Pol_n (\ab{A})$
and for every $\vb{a} \in A^n$, there exists
$\vb{b}$ such that $f^{\ab{A}} (\vb{a})
= f^{\ab{A}} (\vb{b})$ and $\vb{b}$ has at most
$d_{\ab{A}}$ components that are different
from $0$. Hence the equation $f(\vb{x}) \approx y$
has a solution if and only if it has a solution
with at most $d_{\ab{A}}$ nonzero entries.
Thus for the algebra $\ab{A}$, testing only vectors with
at most $d_{\ab{A}}$ nonzero entries is an algorithm,
which, given an equation $f (\vb{x}) \approx y$ of length
$n$, takes at most $c(\vb{A}) \cdot n^{d_\ab{A} + 1}$ many steps
to find whether this equation is solvable:
there are at most
$\sum_{i=0}^{d_{\ab{A}}} {n \choose i} (|A| - 1)^i \le
c_{1} ({\ab{A}}) \cdot  n^{d_{\ab{A}}}$
 many evaluations to be done, each of them
 taking at most $c_{2} (\ab{A}) \cdot n$ many steps.
 The number $d_{\ab{A}}$ in \cite{Ho:TCOT}
 is obtained from Ramsey's Theorem and therefore
 rather large.
In \cite{Ko:TESP}, it is  proved
 that for every finite supernilpotent algebra in a congruence
 modular variety, such a number $d_{\ab{A}}$ exists, again
 using Ramsey's Theorem.
 For rings, lower values
 of $d_{\ab{A}}$ have been obtained in \cite{KS:EOPO} (cf. \cite{KS:TCOT}).
 In \cite{Fo:TCOT17, Fo:TCOT}, A.\ F\"oldv\'{a}ri provides polynomial
 time algorithms for solving equations over finite nilpotent groups and rings
 relying on the structure theory of these algebras.
  In this paper, we extend the method developed in \cite{KS:EOPO}
 from finite nilpotent rings to arbitrary finite supernilpotent
 algebras in congruence modular varieties. For such algebras,
 we compute $d_{\ab{A}}$ as  $|A|^{\log_2 (\mu) + \log_2 (|A|) + 1}$ (Theorem~\ref{thm:snp}).
 The technique that allows to generalize K{\'a}rolyi's  and Szab\'{o}'s method
 is the coordinatization of nilpotent algebras of prime power order
 by elementary abelian
 groups from \cite[Theorem~4.2]{Ai:BTFS}.
 The method generalizes to systems of equations:
 we show for a given finite supernilpotent
 algebra $\ab{A}$  in a congruence modular variety, and a given $s \in \N_0$,
 there is a polynomial time algorithm to test whether a system of at most $s$
 polynomial equations over $\ab{A}$  has a solution.
 If $s$ is not fixed in advance, then \cite[Corollary~3.13]{LZ:TTCS}
 implies that if $\ab{A}$ is not abelian,
 ${\textsc{PolSysSat} (\ab{A})}$ is $\NP$-complete.
 
 Let us finally explain to which class of algebras our results applies:
 A finite algebra $\ab{A}$ from a congruence modular variety
 with finitely many fundamental operations is supernilpotent
 if and only if it is a direct product of nilpotent algebras
 of prime power order; modulo notational differences
 explained, e.g., in \cite[Lemma~2.4]{Ai:BTFS},
 this result has been proved in
 \cite[Theorem~3.14]{Ke:CMVW}.
 Such an algebra is therefore always nilpotent,
 has a Mal'cev term (cf. 
 \cite[Theorem~6.2]{FM:CTFC}, \cite[Theorem~2.7]{Ke:CMVW}),
 and hence generates a congruence permutable variety.
 For a more detailed introduction to supernilpotency
 and, for $k \in \N$, to $k$-supernilpotency, we refer
 to \cite{AM:SAOH, AMO:COTR, Ai:BTFS}.
 \section{A theorem of K{\'a}rolyi and Szab\'{o}}
 In this section, we state a special case of \cite[Theorem~3.1]{KS:EOPO}.
 Since their result is much more general than needed for our purpose, we
 also include a self-contained proof, which is a reduction
 K{\'a}rolyi's and Szab\'{o}'s proof to the case of elementary abelian groups.

For $n \in \N = \{1,2,3,\ldots\}$, we denote the set $\{1,2,\ldots, n\}$ by $\uli{n}$.
Let $A$ be a set with an element $0 \in A$, and let $J \subseteq \uli{n}$.
For $\vb{a} \in A^n$, $\vbr{a}{J}$ is defined by $\vbr{a}{J} \in A^n$, $\vbr{a}{J} (j) = \vb{a} (j)$ for
$j \in J$ and $\vbr{a}{J} (j) = 0$ for $i \in \uli{n} \setminus J$.
Suppose that $1$ is an element of $A$. Then by $\vb{1}$, we denote the vector
$(1,1,\ldots,1)$ in $A^n$, and for $J \subseteq \uli{n}$, $\vbr{1}{J}$ is the
vector $(v_1, \ldots, v_n)$ with $v_j = 1$ if $j \in J$ and $v_j = 0$ if $j \not\in J$.
For any sets $C,D$, we write $C \subset D$ for ($C \subseteq D$ and $C \neq D$).

We first need the following variation of \cite[Theorem~1]{Br:CTWR} and
\cite[Theorem~3.2]{KS:EOPO}, which is proved
using several arguments from the proof of \cite[Theorem~3.1]{Al:CN} and
from \cite{Br:CTWR}.
\begin{lemma} \label{lem:polies}
  Let $F$ be a finite field, let $k, m, n \in \N$,  let $q := |F|$,
  let $p_1,\ldots, p_m \in F[x_1,\ldots, x_n]$ be polynomials
  such that for each $i \in \uli{m}$, 
  each monomial of $p_i$ contains at most $k$ variables.
  Then there exists $J \subseteq \uli{n}$ such that
  $|J| \le k m (q-1)$ and $p_i (\vbr{1}{J}) = p_i (\vb{1})$
  for all $i \in \uli{m}$.
\end{lemma}
\begin{proof}
We proceed by induction on $n$. If $n \le km(q-1)$, then
we take $J := \uli{n}$.
For the induction step, we assume that $n > km(q-1)$.

We first produce a set $J_1 \subset \uli{n}$ such that
$p_i (\vbr{1}{J_1}) = p_i (\vb{1})$ for all $i \in \uli{m}$.
Seeking a contradiction, we suppose that no such $J_1$
exists. Following an idea from the proof of \cite[Theorem~3.1]{Al:CN},
we consider the polynomials
\[
\begin{array}{rcl}
    q_1 (x_1,\ldots,x_n) & := &  \prod_{i=1}^m (
    1 - (p_i (\vb{x}) - p_i (\vb{1}))^{q-1}
                                         ), \\ 
    q_2(x_1,\ldots,x_n) & := &  x_1 x_2 \cdots x_n - q_1 (x_1,\ldots,x_n).
\end{array}
\]
We first show that for all $\vb{a} \in \{0,1\}^n$, $q_2( \vb{a} ) = 0$.
To this end, we first consider the case $\vb{a} = \vb{1}$.
Then
\(
q_2 (\vb{a}) = 1 - \prod_{i=1}^m 1 = 0.
\)
If $\vb{a} \in \{0,1\}^n \setminus \{\vb{1}\}$, then
by the assumptions, there is $i \in \uli{m}$
such that $p_i (\vb{a}) \neq p_i (\vb{1})$.
Then $1 - (p_i (\vb{a}) - p_i (\vb{1}))^{q-1} = 0$.
Therefore $q_2 (\vb{a}) = 0$.
Hence the polynomial $q_2$ vanishes at $\{0,1\}^n$. By the Combinatorial
Nullstellensatz \cite[Theorem~1.1]{Al:CN} applied to $g_j (x_j) := x_j^2 - x_j$,
$q_2$
then lies in the ideal $V$ of $F[x_1, \ldots, x_n]$ generated by
$G = \{x_j^2 - x_j \mid j \in \uli{n}\}$. Hence
$x_1x_2 \cdots x_n - q_1 (x_1, \ldots, x_n) \in V$. Since
the leading monomials of the polynomials in $G$ are coprime,
$G$ is a Gr\"obner basis of $V$
(with respect to $x_1 > x_2 > \dots > x_n$, lexicographic order, cf. \cite[p.337]{Ei:CA}).
Therefore,
reducing $q_1 (x_1,\ldots,x_n)$ modulo $G$, we must obtain $x_1x_2\cdots x_n$ as
the remainder (as defined, e.g., in \cite[p.334]{Ei:CA}).
Because of the form of all polynomials in $G$
(all variables of $g_j$ occur in the leading term of $g_j$),
none of the reduction steps
increases the number of variables in any monomial. Therefore, 
$q_1 (x_1, \ldots, x_n)$ must contain a monomial that contains all $n$ variables.
Computing the expansion of $q_1$ by multiplying out all products from
its definition, we see
that each monomial in $q_1$ contains at most $km(q-1)$ variables.
Hence $n \le km(q-1)$, which contradicts the assumption $n > km(q-1)$.
This contradiction shows that there is set $J_1 \subset \uli{n}$ such that
$p_i (\vbr{1}{J_1}) = p_i (\vb{1})$ for all $i \in \uli{m}$.
Now we let ${n'} := |J_1|$, and we assume that
$J_1 = \{j_1,\ldots, j_{n'}\}$ with $j_1 < \dots < j_{n'}$.
For $i \in \uli{m}$, we define $p'_i \in F[y_1, \ldots, y_{n'}]$ by
\[
p'_i (x_{j_1}, \ldots, x_{j_{n'}}) = p_i (\vbr{x}{J_1}).
\]
By the induction hypothesis, there exists $J_2 \subseteq \uli{{n'}}$ with
$|J_2| \le k m (q-1)$ such that
$p'_i (\vbr{1}{J_2}) = p'_i (\vb{1})$ for all $i \in \uli{m}$.
Now we define
$J := \{j_t \mid t \in J_2\}$. We have $J \subseteq J_1$, and therefore
$\vbr{1}{J} = \vbr{(\vbr{1}{\mathit{J}})}{J_1}$.
Then $p_i (\vbr{1}{J}) =
      p_i (\vbr{(\vbr{1}{\mathit{J}})}{J_1}) =
      p'_i ( \vbr{1}{J} (j_1), \ldots, \vbr{1}{J} (j_{n'}))
      = p'_i ( \vbr{1}{J_2} ) = p'_i (\vb{1}) =
      p_i (\vbr{1}{J_1}) = p_i (\vb{1})$, which completes the induction step.  \end{proof}

      We will need the following special case of \cite[Theorem~3.1]{KS:EOPO}. 
Let $\Pot_k (\uli{n})$ denote the set
 $\{I \subseteq \uli{n} : |I| \le k \}$ of subsets of $\uli{n}$ with at most $k$ elements.
 \begin{theorem}[cf. {\cite[Theorem~3.1]{KS:EOPO}}] \label{thm:ks}
    Let $n \in \N$, let $k \in \N_0$, let $p$ be a prime,
    and let $m \in \N$.
    Let $\varphi : \Pot_k (\uli{n}) \to \Z_p^m$. Then there is
    $U \subseteq \uli{n}$ with $|U| \le km(p-1)$ such that
    \[
    \sum_{J \in \Pot_k(\uli{n})} \varphi (J) = \sum_{J \in \Pot_k (U)} \varphi (J).
    \]
\end{theorem}
 \begin{proof}
We denote the vector $\varphi (J)$ by $( (\varphi (J))_1, \ldots, (\varphi (J))_m)$, and
we define $m$ polynomial functions $f_1, \ldots, f_m \in \Z_p[x_1, \ldots, x_n]$
by
\[
f_i (x_1, \ldots x_n) := \sum_{J \in \Pot_k (\uli{n})} \big( (\varphi (J))_i \cdot \prod_{j \in J} x_j \big).
\]
for $i \in \uli{m}$.
By Lemma~\ref{lem:polies}, there is a subset  $U$ of $\uli{n}$
with $|U| \le k m (p-1)$ such that
for all $i \in \uli{m}$, we have
$f_i (\vb{1}) = f_i (\vbr{1}{U})$. Hence
$\sum_{J \in \Pot_k (\uli{n})} (\varphi (J))_i =
f_i (\vb{1}) = f_i (\vbr{1}{U}) =  \sum_{J \in \Pot_k (\uli{n}), J \subseteq U} (\varphi (J))_i
 = \sum_{J \in \Pot_k (U)} (\varphi (J))_i$.  \end{proof}

 \section{Absorbing components}

Let $A$ be a set, let $0_A$ be an element of $A$, let $\ab{B} = \algop{B}{+,-,0}$
be an abelian group, let $n \in \N$, let $f : A^n \to B$, and let $I \subseteq \uli{n}$.
By $\Dep (f)$ we denote the set $\{ i \in \uli{n} \mid f \text{ depends on its $i$\,th argument}\}$.
We say that $f$ is \emph{absorbing in its $j$\,th argument} if for all $\vb{a}= (\vb{a}(1), \ldots, \vb{a}(n)) \in A^n$
with $\vb{a} (j) = 0_A$ we have $f(\vb{a}) = 0$. In the sequel, we will denote $0_A$ simply by
$0$.
We say that $f$ is \emph{absorbing in $I$} if $\Dep (f) \subseteq I$ and for every
$i \in I$, $f$ is absorbing in its $i$\,th argument.
\begin{lemma} \label{lem:mondec}
  Let $A$ be a set, let $0$ be an element of $A$, let $\ab{B} = \algop{B}{+,-,0}$
  be an abelian group, let $n \in \N$,  and
  let $f : A^n \to B$. Then there is exactly one sequence $(f_I)_{I \subseteq \uli{n}}$ of functions from $A^n$ to $B$  such that
  for each $I \subseteq \uli{n}$, $f_I$ is absorbing in $I$ and $f = \sum_{I \subseteq \uli{n}} f_I$.
  Furthermore, each function $f_I$ lies in the subgroup $\ab{F}$ of $\ab{B}^{A^n}$ that is
  generated by the functions $\vb{x} \mapsto f (\vbr{x}{I})$, where
  $I \subseteq \uli{n}$.
\end{lemma}
\begin{proof} We first prove the existence of such a sequence.
         To this end, we define $f_I$ by recursion on $|I|$.
         We define $f_{\emptyset} (\vb{a}) := f (0,\ldots, 0)$ and for $I \neq \emptyset$,
         we let
         \[
            f_I (\vb{a}) := f (\vbr{a}{I}) - \sum_{J \subset I} f_J (\vb{a}).
         \]
         By induction on $|I|$, we see that $\Dep (f_I) \subseteq I$ and
         that $f_I$ lies in the subgroup $\ab{F}$.
         We will
         now show that each $f_I$ is absorbing in $I$, and we again proceed
         by induction on $|I|$. Let $i \in I$, and let $\vb{a} \in A^n$ be
         such that $\vb{a} (i) = 0$. We have to show $f_I (\vb{a}) = 0$. We compute
         $f_I (\vb{a}) = f (\vbr{a}{I}) - \sum_{J \subset I} f_J (\vb{a})$.
         By the induction hypothesis, we have $f_J (\vb{a}) = 0$ for those $J$ with
         $i \in J$. Hence $f (\vbr{a}{I}) - \sum_{J \subset I} f_J (\vb{a}) =
         f (\vbr{a}{I}) - \sum_{J \subseteq I \setminus \{i\}} f_J (\vb{a})$, and because
         of $\vbr{a}{I} = \vbr{a}{I \setminus \{i\}}$, this is equal to
         $f (\vbr{a}{I \setminus \{i\}}) - \sum_{J \subseteq I \setminus \{i\}} f_J (\vb{a}) =
         f (\vbr{a}{I \setminus \{i\}}) - \sum_{J \subset I \setminus \{i\}} f_J (\vb{a}) - f_{I \setminus \{i\}} (\vb{a})$.
         By the definition of $f_{I \setminus \{i\}}$, the last expression is equal to
         $f_{I \setminus \{i\}} (\vb{a}) - f_{I \setminus \{i\}} (\vb{a}) = 0$. This completes the induction
          proof; hence each $f_I$ is absorbing in $I$.
         In order to show $f = \sum_{I \subseteq \uli{n}} f_I$, we choose $\vb{a} \in A^n$ and
         compute $\sum_{I \subseteq \uli{n}} f_I (\vb{a}) = f_{\uli{n}} (\vb{a}) + \sum_{I \subset \uli{n}} f_I (\vb{a})
         = f (\vbr{a}{\uli{n}}) - \sum_{J \subset \uli{n}} f_J (\vb{a}) + \sum_{I \subset \uli{n}} f_I (\vb{a})
         = f (\vb{a})$.
         This completes the proof of the existence of such a sequence.

         For the uniqueness, assume that $f = \sum_{I \subseteq \uli{n}} f_I = \sum_{I \subseteq \uli{n}} g_I$ and
           that for all $I$, $f_I$ and $g_I$  are absorbing in $I$. We show by induction on $|I|$ that
           $f_I = g_I$. Let $I := \emptyset$.
           First we notice that $f (0, \ldots, 0) = \sum_{J \subseteq \uli{n}} f_J (0,\ldots, 0) =
           \sum_{J \subseteq \uli{n}} g_J (0,\ldots, 0)$.
            Since $f_J$ and $g_J$ are absorbing, the summands with $J \neq \emptyset$ are $0$,
           and thus $f_{\emptyset} (0, \ldots, 0) =  \sum_{J \subseteq \uli{n}} f_J (0,\ldots, 0) =
           f(0,\ldots,0) = \sum_{J \subseteq \uli{n}} g_J (0,\ldots, 0) = g_{\emptyset} (0,\ldots,0)$.
           Since both $f_{\emptyset}$ and $g_{\emptyset}$ are constant functions, they are equal.
           For the induction step, we assume $|I| \ge 1$.
           Let $\vb{a} \in A^n$.
           Then $\sum_{J \subseteq \uli{n}} f_J (\vbr{a}{I}) = \sum_{J \subseteq \uli{n}} g_J (\vbr{a}{I})$.
           Only the summands with $J \subseteq I$ can be nonzero, and therefore
           $\sum_{J \subseteq I} f_J (\vbr{a}{I}) = \sum_{J \subseteq I} g_J (\vbr{a}{I})$.
           By the induction hypothesis, $f_J = g_J$ for $J \subset I$. Therefore,
           $f_I (\vbr{a}{I}) = g_I (\vbr{a}{I})$. Since $f_I$ and $g_I$ depend only on the arguments
           at positions in $I$, we obtain $f_I (\vb{a}) = f_I (\vbr{a}{I}) = g_I (\vbr{a}{I})
           = g_I (\vb{a})$.
           Thus $f_I = g_I$.  \end{proof}

           Actually, the component $f_I$ can be computed by
           $f_I (\vb{a}) = \sum_{J \subseteq I} (-1)^{|I| + |J|} f(\vbr{a}{J})$.
           \begin{definition}
             Let $A$ be a set, let $0$ be an element of $A$, let $\ab{B} = \algop{B}{+,-,0}$
             be an abelian group, let $n \in \N$, let $f : A^n \to B$, and let $J \subseteq \uli{n}$.
             Then we call the sequence $(f_I)_{I \subseteq \uli{n}}$ such that
             for each $I \subseteq \uli{n}$, $f_I$ is absorbing in $I$, and $f = \sum_{I \subseteq \uli{n}} f_I$ the
             \emph{absorbing decomposition} of $f$, and $f_J$ the \emph{$J$-absorbing component} of $f$.
             We define the \emph{absorbing degree of $f$} by
             $\adeg (f) := \max \, (\{-1\} \cup \{ |J| :  J \subseteq \uli{n} \text{ and } f_J \neq 0 \})$. 
           \end{definition}

           \begin{theorem} \label{thm:redweight}
              Let $A$ be a set, let $0$ be an element of $A$,
              let $p$ be a prime, let $k \in \N_0$, let $n \in \N$,  and let  $f_1, \ldots, f_m : A^n \to \Z_p$.
              We assume that each $f_i$ is of absorbing degree at most $k$.
              Let $\vb{a} \in A^n$. Then there is $U$ with $|U| \le km(p-1)$ such that
              for all $i \in \uli{m}$, we have $f_i (\vb{a}) = f_i (\vbr{a}{U})$.
           \end{theorem}
           \begin{proof}
            We define a function $\varphi : \Pot_k (\uli{n}) \to \Z_p^m$ by
            $\varphi (J) := ((f_1)_J (\vb{a}), \ldots, (f_m)_J (\vb{a}))$, where
            for $i \in \uli{m}$, $\bigl((f_i)_J\bigr)_{J \subseteq \uli{n}}$
            is the absorbing
            decomposition of $f_i$. 
             Then Theorem~\ref{thm:ks} yields a subset $U$ of $\uli{n}$ with
             $|U| \le  km(p-1)$ such that
             $\sum_{J \in \Pot_k (\uli{n})} \varphi (J) = \sum_{J \in \Pot_k (U)} \varphi (J)$.
             Since $(f_i)_J = 0$  for all $J$ with $|J| > k$,
             we have
          $\sum_{J \in \Pot_k (\uli{n})} \varphi (J) =
                \sum_{J \in \Pot_k (\uli{n})} ((f_1)_J (\vb{a}), \ldots, (f_m)_J (\vb{a})) \\ =
              \sum_{J \subseteq \uli{n}} ((f_1)_J (\vb{a}), \ldots, (f_m)_J (\vb{a}))
              = (f_1 (\vb{a}), \ldots, f_m (\vb{a}))$ and
              \begin{multline*}
             \sum_{J \in \Pot_k (U)} \varphi (J) = \sum_{J \in \Pot_k (U)} ((f_1)_J (\vb{a}), \ldots, (f_m)_J (\vb{a})) \\ = \sum_{J \subseteq U} ((f_1)_J (\vb{a}), \ldots, (f_m)_J (\vb{a}))
             =  \sum_{J \subseteq U} ((f_1)_J (\vbr{a}{U}), \ldots, (f_m)_J (\vbr{a}{U}))
            \\ =  \sum_{J \subseteq \uli{n}} ((f_1)_J (\vbr{a}{U}), \ldots, (f_m)_J (\vbr{a}{U}))
              =  (f_1 (\vbr{a}{U}), \ldots, f_m (\vbr{a}{U})).
              \end{multline*}
           \end{proof}
             
             \section{Polynomial mappings}

             In this section, we develop a property of polynomial mappings
             of finite supernilpotent algebras in congruence modular varieties.
             We call an algebra $\ab{A} = \algop{A}{+,-,0, (f_i)_{i \in S}}$ an
             \emph{expanded group} if its reduct $\ab{A}^+ = \algop{A}{+,-,0}$
             is a group, an \emph{expanded abelian group} if
             $\ab{A}^+$ is an abelian group, and an
             \emph{expanded elementary abelian group} if $\ab{A}^+$ is
             elementary abelian, meaning that $\ab{A}^+$  is abelian and
             all its nonzero elements have the same prime order.
              For an algebra $\ab{A}$ and $n,s \in \N$, we define the set $\Pol_{n,s} (\ab{A})$
              of
             polynomial maps from $\ab{A}^n$ to $\ab{A}^s$ as the set
             of all mappings $\vb{a} \mapsto (f_1 (\vb{a}), \ldots, f_s (\vb{a}))$ with
             $f_1, \ldots, f_s \in \Pol_n (\ab{A})$.

             \begin{lemma} \label{lem:snpabs}
               Let $k, n \in \N$, let $\ab{A}$ be a $k$-supernilpotent
               expanded abelian group, and let $f \in \Pol_n (\ab{A})$.
               Then $f$ is of
               absorbing degree at most $k$.
             \end{lemma}
             \begin{proof}
             Let $J \subseteq \uli{n}$ with $|J| > k$, and let
             $f_J$ be the $J$-absorbing component of $f$.
             Let $m := |J|$ and let $J = \{ i_1, \ldots, i_m \}$.
             Using Lemma~\ref{lem:mondec}, we obtain that the
             function $g : A^m \to A$ defined by
             $g (a_{i_1}, \ldots, a_{i_m}) := f_J (\vb{a})$ for $\ab{a} \in A^n$
             is an absorbing function in $\Pol_m (\ab{A})$. Hence
             \cite[Lemma~2.3]{Ai:BTFS} and the remark immediately
             preceding that Lemma yield that $g$ is the zero function.
             Thus $f_J = 0$. Hence the absorbing degree of $f$ is at most $k$.
              \end{proof}

             We first consider polynomial mappings of supernilpotent
             expanded elementary abelian groups of prime power order.
                \begin{theorem} \label{thm:sysp}
                  Let $k,n,s, \alpha \in \N$,
                  let $p \in \mathbb{P}$,
                  and let
                  $\ab{A}
                  $ be a $k$-su\-per\-nil\-potent
               expanded elementary abelian group of
               order $p^{\alpha}$. 
               Let $F = (f_1, \ldots, f_s) \in \Pol_{n,s} (\ab{A})$, and let $\vb{a} \in A^n$.
               Then there is $U \subseteq \uli{n}$ with $|U| \le k s \alpha (p-1)$ such that
               $F (\vb{a}) = F (\vbr{a}{U})$.
             \end{theorem}
                \begin{proof}
                We let $\pi$ be a group isomorphism from $\algop{A}{+,-,0}$ to
                $\Z_p^{\alpha}$, and for $a \in A$, we denote
                $\pi (a)$ by $(\pi_1 (a), \ldots, \pi_{\alpha} (a))$.
             For each $r \in \uli{s}$ and each $\beta  \in \uli{\alpha}$,
             let $f_{r,\beta} : A^n \to \Z_p$ be defined by
             $f_{r, \beta} (\vb{a}) = \pi_{\beta} (f_r (\vb{a}))$; hence $f_{r, \beta} (\vb{a})$ 
             is the $\beta$\,th component of
             $f_r (\vb{a})$. Since $f_r \in \Pol_n (\ab{A})$ and
             $\ab{A}$ is $k$-supernilpotent, Lemma~\ref{lem:snpabs} implies that
             each of these $f_{r, \beta}$
             is of absorbing degree at most $k$.
             Setting $m := s \alpha$, Theorem~\ref{thm:redweight} yields
             $U$ with $|U| \le k s \alpha (p-1)$ such that $f_{r, \beta} (\vb{a}) = f_{r, \beta} (\vbr{a}{U})$
             for all $r \in \uli{s}$ and $\beta \in \uli{\alpha}$.
              Then clearly $F (\vb{a}) = F (\vbr{a}{U})$.  \end{proof}

               We apply this result to  polynomial mappings of \emph{direct products}
                    of finite supernilpotent
                    expanded elementary abelian groups.
              For a vector $\vb{a} \in A^n$, we call the number of its nonzero entries
              the \emph{weight} of $\ab{a}$; formally,
              \(
                    \wt (\vb{a}) :=  |\{ j \in \uli{n} : \vb{a} (j) \neq 0 \}|.
                    \)
            \begin{theorem} \label{thm:sysallg}
              Let $n,s,t, k_1, \ldots, k_t  \in \N$.
              For 
              each $i \in \uli{t}$, 
              let $\ab{B}_i$ a $k_i$-supernilpotent
              expanded elementary abelian group with $|\ab{B}_i| = p_i^{\alpha_i}$,
              where $p_i$ is a prime and $\alpha_i \in  \N$.
              Let $\ab{A} := \prod_{i=1}^t \ab{B}_i$, let
              $F \in \Pol_{n,s} (\ab{A})$, and let $\vb{a} \in A^n$.
               Then there is $\vb{y} \in A^n$ with
               \( \wt (\vb{y}) \le \sum_{i = 1}^t k_i s \alpha_i (p_i-1) \) such that
               $F (\vb{a}) = F (\vb{y})$.
              \end{theorem}
            \begin{proof}
              For $i \in \uli{t}$, let
              $\nu_i$ be the $i$\,th projection kernel.
              Applying Theorem~\ref{thm:sysp} to
              $\ab{A}/\nu_i$, which is isomorphic to $\ab{B}_i$,  and
              $\vb{b} := \vb{a}/\nu_i$, we obtain
              $U_i \subseteq \uli{n}$ with $|U_i| \le k_i s \alpha_i (p_i - 1)$ such that
              $F^{\ab{A} / \nu_i} (\vbr{b}{U_i}) = F^{\ab{A}/\nu_i} (\vb{b})$.
              Lifting $\vbr{b}{U_i}$ to $A$, we obtain $(x_{i,1}, \ldots, x_{i, n}) \in A^n$ such that
              $(x_{i,1}, \ldots, x_{i, n}) / \nu_i = \vbr{b}{U_i}$ and
              $x_{i,j} = 0$ for $j \in \uli{n} \setminus U_i$.
              Now for every $j \in \uli{n}$, we define $y_j \in A$
              by the equations
              \[
              y_j \equiv_{\nu_i} x_{i, j} \text{ for all } i \in \uli{t}.
              \]
              For each $i \in \uli{t}$, we have $F (y_1, \ldots, y_n) / {\nu_i}
              = F^{\ab{A}/\nu_i} ( x_{i,1} / \nu_i, \ldots, x_{i,n} / \nu_i ) =
              F^{\ab{A}/\nu_i} ( \vbr{b}{U_i} ) =
              F^{\ab{A}/\nu_i} ( \vb{b} ) =
              F^{\ab{A}/\nu_i} ( \vb{a}/\nu_i ) = F (\vb{a}) / \nu_i$.
              Hence $F (\vb{y}) = F (\vb{a})$.
                  For $j \in \uli{n} \setminus (U_1 \cup \dots \cup U_t)$,
                  and for all $i \in \uli{t}$, we have $x_{i,j} = 0$, and therefore
                  $y_j = 0$. Hence the number of nonzero entries in $\vb{y}$ is
                  at most $\sum_{i=1}^t |U_i| = \sum_{i = 1}^t k_i s \alpha_i (p_i-1)$.  \end{proof}

                  Now we consider \emph{arbitrary} finite supernilpotent algebras
                  in congruence modular varieties. In these algebras, 
                  we can introduce group operations preserving nilpotency using \cite{Ai:BTFS}.
                     \begin{lemma} \label{lem:exexp}
                      Let $\mu \in \N$, 
          let $\ab{A} = \algop{A}{(f_i)_{i \in S}}$ be a finite supernilpotent algebra in a congruence
          modular variety all of whose fundamental operations have arity at most $\mu$, and let $z \in A$.
          Let $t \in \N_0$, let $p_1, \ldots, p_t$ be different primes, and  let $\alpha_1, \ldots, \alpha_t \in \N$ such that $|\ab{A}| = \prod_{i=1}^t p_i^{\alpha_i}$.
          For $i \in \uli{t}$, let
          \(
          k_i := (\mu (p_i^{\alpha_i} - 1))^{\alpha_i - 1}.
          \)
          Then there are operations $+$ (binary), $-$ (unary), $0$ (nullary) on $A$ such that
          $\ab{A'} = \algop{A}{+,-,0, (f_i)_{i \in S}}$ 
          is isomorphic to
          a direct product
          $\prod_{i=1}^t \ab{B}_i'$, where each $\ab{B}_i'$ is a
          $k_i$-supernilpotent expanded elementary abelian group,
          and $0^{\ab{A'}} = z$.
                  \end{lemma}
                  \begin{proof}
                  Since the result is true for $|A| = 1$, we henceforth assume $|A| \ge 2$.
                  By \cite{Ke:CMVW}, $\ab{A}$ is isomorphic to a direct product
                  $\prod_{i=1}^t \ab{B}_i$ of nilpotent algebras of prime power order.
                  We let $(\pi_1 (a), \ldots, \pi_t(a))$ denote
                  the image of $a$ of the underlying isomorphism.
                  As a finite supernilpotent algebra in a congruence
                  modular variety, $\ab{A}$ is nilpotent (cf. \cite[Lemma~2.4]{Ai:BTFS}) and
                  therefore has a Mal'cev term \cite[Theorem~6.2]{FM:CTFC}.
                  We use
                  \cite[Theorem~4.2]{Ai:BTFS} to expand each $\ab{B}_i$ with operations
                  $+_i$ and $-_i$ such that
                  the expansion $\ab{B}_i'$ is a nilpotent expanded elementary abelian group with
                  zero element $\pi_i (z)$. By \cite[Theorem~1.2]{Ai:BTFS},
                  $\ab{B}_i'$ is $k_i$-supernilpotent.  \end{proof}

                  We note that the supernilpotency degree of $\ab{A}'$ may be strictly larger
                  than the supernilpotency degree of $\ab{A}$.

                                  Combining these results, we obtain the following result on polynomial mappings
                                  on arbitrary finite supernilpotent algebras in congruence modular
                                  varieties.
       \begin{theorem} \label{thm:snp}
          Let $\mu \in \N$,           
          let $\ab{A}$ be a finite supernilpotent algebra in a congruence
          modular variety all of whose fundamental operations have arity at most $\mu$.
          Let $p_1, \ldots, p_t$ be distinct primes, and let $\alpha_1, \ldots, \alpha_t \in \N$ such
          that $|\ab{A}| = \prod_{i=1}^t p_i^{\alpha_i}$.
          Let $F \in \Pol_{n,s} (\ab{A})$ be a polynomial
          map from $A^n$ to $A^s$, and let $z \in A$.
          Then for every $\vb{a} \in A^n$ there is $\vb{y} \in A^n$
          such that $F (\vb{y}) = F (\vb{a})$ and
          \(
          |\{ j \in \uli{n} : \vb{y} (j) \neq z\}| \le
                s \sum_{i = 1}^t (\mu (p_i^{\alpha_i} - 1))^{\alpha_i - 1} \alpha_i (p_i-1)
        \le       s \mu^{-1} |A|^{\log_2 (\mu) + \log_2 (|A|)} \log_2 (|A|) \le s |A|^{\log_2 (\mu) + \log_2 (|A|) + 1}.
          \)
        \end{theorem}
       \begin{proof} Let $\ab{A'} = \prod_{i=1}^t \ab{B}_i'$ be the expansion of $\ab{A}$ produced by
       Lemma~\ref{lem:exexp}, and for each $i \in \uli{t}$, let $p_i \in \mathbb{P}$ and $\alpha_i \in \N$ be such that
       $|\ab{B}_i| = p_i^{\alpha_i}$.
       Clearly, $F$ is a also a polynomial map of
       $\ab{A'}$.
       Let $k_i  = (\mu (p_i^{\alpha_i} - 1))^{\alpha_i - 1}$.
       Then Theorem~\ref{thm:sysallg} yields $\vb{y} \in A^n$ such that
       $|\{ j \in \uli{n} : \vb{y} (j) \neq z \}|
       \le \sum_{i = 1}^t k_i s \alpha_i (p_i-1)$.
       Using the obvious estimate $\alpha_i \le \log_2 (|A|)$, we obtain
    \begin{multline*}
      \sum_{i = 1}^t k_i s \alpha_i (p_i-1)  = 
        s \sum_{i=1}^t (\mu (p_i^{\alpha_i} - 1))^{\alpha_i - 1} \log_2 (|A|)  (p_i - 1) \\ \le 
        s \log_2 (|A|) \sum_{i=1}^t \mu^{\alpha_i - 1} (p_i^{\alpha_i})^{\alpha_i -1} p_i^{\alpha_i} \le 
        s \log_2 (|A|) \sum_{i=1}^t \mu^{\log_2(|A|) - 1} (p_i^{\alpha_i})^{\alpha_i}   \\\le
        s \mu^{\log_2  (|A|) - 1} \log_2 (|A|) \sum_{i=1}^t (p_i^{\alpha_i})^{\log_2 (|A|)} \le 
        s \mu^{\log_2  (|A|) - 1} \log_2 (|A|) ( \sum_{i=1}^t p_i^{\alpha_i} )^{\log_2 (|A|)} \\  \le
        s \mu^{\log_2  (|A|) - 1} \log_2 (|A|) ( \prod_{i=1}^t p_i^{\alpha_i} )^{\log_2 (|A|)} \le
        s \mu^{\log_2  (|A|) - 1} \log_2 (|A|) |A|^{\log_2 (|A|)}  \\ =
        s \mu^{-1} |A|^{\log_2 (\mu) + \log_2 (|A|)} \log_2 (|A|)  \le
        s |A|^{\log_2 (\mu) + \log_2 (|A|) + 1}.
      \end{multline*}
        \end{proof}
       \section{Systems of equations}
       We will now explain how these results give a polynomial time algorithm
       for solving systems of a fixed number of equations over the finite
       supernilpotent algebra $\ab{A}$. The size $m$ of a system of
       polynomial equations is measured
       as the length of the polynomial terms used to represent the system.
       For measuring the ``running time'' of our algorithm, we count
       the number of $\ab{A}$-operations: each such $\ab{A}$-operation, may,
       for example, be done by looking up one value in the operation tables
       defining $\ab{A}$.
        \begin{theorem} \label{thm:polsyssat}
          Let $\ab{A}$ be a finite supernilpotent algebra in a congruence modular
          variety all of
          whose fundamental operations are of arity at most $\mu$,
          and let $s \in \N$. We consider the following algorithmic problem
          $s$-$\textsc{PolSysSat} (\ab{A})$:
          \begin{quote}
            \textbf{Given:} $2s$ polynomial terms $f_1, g_1, \ldots, f_s, g_s$ over $\ab{A}$. \\
            \textbf{Asked:} Does the system $f_1 \approx g_1, \ldots, f_s \approx g_s$ have a solution
            in~$\ab{A}$?
          \end{quote}
          Let $m$ be the length of the input of this system, and 
          let
         \[
         e := s |A|^{\log_2 (\mu) + \log_2 (|A|) + 1} + 1.
         \]
         Then we can decide $s$-$\textsc{PolSysSat} (\ab{A})$ using at most
         $O(m^{e-1})$ evaluations of all terms occuring in the system.
         Therefore, we have an algorithm
         that determines whether a system of $s$ polynomial equations over $\ab{A}$ 
         has a solution using $O(m^e)$ many $\ab{A}$-operations.
       \end{theorem}
        \begin{proof}
          Let $n$ be the number of different variables that occur in the
          given system. We may assume that these variables
          are $x_1,\ldots, x_n$, and that our system is
          $\bigwedge_{i=1}^s f_i (x_1, \ldots, x_n) \approx g_i (x_1, \ldots, x_n)$.
          We choose an element $z \in A$, and we 
          will show: if this system has a solution in $\vb{a} \in A^n$,
          then it has a solution in 
              \[
           C :=
           \{ \vb{y} \in A^n : |\{ j \in \uli{n} : \vb{y} (j) \neq z \}| \le e - 1 \}.
       \]
        For proving this claim, we first observe that
        $\ab{A}$ is a finite nilpotent algebra in a congruence modular variety,
       and it therefore has a Mal'cev term $d$.
        We consider the polynomial map $H = (h_1, \ldots, h_s)$, where
        $h_i (\vb{x}) := d(f_i (\vb{x}), g_i (\vb{x}), z)$ for $i \in \uli{s}$ and $\vb{x} \in A^n$.
        Since $\vb{a}$ is a solution of the system, $H(\vb{a}) = (z,z,\ldots, z)$.
        By Theorem~\ref{thm:snp}, there is $\vb{y} \in C$ such that
        $H(\vb{y}) = H(\vb{a})$.
        Then for every $i \in \uli{s}$, we have
        $d (f_i (\vb{y}), g_i (\vb{y}), z) = z$.
        By \cite[Corollary~7.4]{FM:CTFC}, the function $x \mapsto d(x,g_i (\vb{y}),z)$ is injective.
        Since $d(f_i (\vb{y}), g_i (\vb{y}), z) =
        z = d(g_i (\vb{y}), g_i (\vb{y}), z)$, this injectivity implies that $f_i (\vb{y}) = g_i (\vb{y})$.
        Hence $\vb{y}$ is a solution that lies in $C$.

        The algorithm for solving the system now simply evaluates the system at all places in
        $C$; if a solution is found, the answer is ``yes''. If we find no solution inside $C$ , we answer ``no'',
        and by the argument above, we know that in this case, the system has no solution inside $\ab{A}^n$ at all.

        We now estimate the complexity of this procedure:
       There is a $c \in \N$ such that for all $n \in \N$, $|C| \le c_1 n^{e-1}$,
       hence we have to do $O(n^{e-1})$ evaluations of all the terms $f_i, g_i$ in the system. 
       Such an evaluation can be done using at most  $O(m)$ many  $\ab{A}$-operations.
            Since the length of the input $m$ is at least the number of variables $n$ occuring in it, this solves $s$-$\textsc{PolSysSat} (\ab{A})$ using at most
            $O(m^e)$ many $\ab{A}$-operations.
        \end{proof}       

       \section{Circuit satisfiability}
       With every finite algebra $\ab{A}$,
       \cite{IK:SIMC} associates a number of computational problems that
       involve circuits whose gates are taken from the fundamental operations of
       $\ab{A}$. One of these problems is $\textsc{SCsat} (\ab{A})$. It takes
       as an input $2s$ circuits $f_1, g_1, \ldots, f_s, g_s$ over $\ab{A}$  with $n$ input
       variables, and asks whether there is a $\vb{a} \in A^n$ such that
       the evaluations at $\vb{a}$ satisfy $f_i (\vb{a}) = g_i (\vb{a})$ for
       all $i \in \uli{s}$. For finite algebras in congruence modular varieties,
       \cite[Corollary~3.13]{LZ:TTCS} implies that
       $\textsc{SCsat} (\ab{A})$ is in $\PP$ when $\ab{A}$ is abelian,
       and $\NP$-complete otherwise.
       However, if we restrict the number $s$ of circuits, we obtain a different
       problem, which we call $s$-$\textsc{SCsat} (\ab{A})$ in the sequel.
       Obviously, $1$-$\textsc{SCsat} (\ab{A})$ is the circuit satisfiability
       problem called $\textsc{Csat} (\ab{A})$ in \cite{IK:SIMC}.
       The method used to prove Theorem~\ref{thm:polsyssat} immediately
       yields:
       \begin{theorem}
         Let $\ab{A}$ be a finite supernilpotent algebra in a congruence
         modular variety, and let $s \in \N$. Then
         $s$-$\textsc{SCsat} (\ab{A})$ is in $\PP$.
       \end{theorem}
       Hence a supernilpotent, but not abelian algebra $\ab{A}$ has
       $s$-$\textsc{SCsat} (\ab{A})$ in $\PP$, whereas
       $\textsc{SCsat}$ is $\NP$-complete.
       In the converse direction, Theorem~9.1 from \cite{IK:SIMC}
       has the following corollary.
       \begin{corollary}
         Let $\ab{A}$ be a finite algebra from a congruence modular variety.
         If $\ab{A}$ has no homomorphic image $\ab{A}'$ such that
         $2$-$\textsc{SCsat} (\ab{A}')$ is $\NP$-complete, then
         $\ab{A}$ is nilpotent.
       \end{corollary}
       \begin{proof}
       Suppose that $\ab{A}$ has a homomorphic image $\ab{A'}$ for
       which $\textsc{Csat} (\ab{A'})$ is $\NP$-complete.
       Then also $2$-$\textsc{SCsat} (\ab{A'})$ is $\NP$-complete
       because an algorithm solving $2$-$\textsc{SCsat}$ can be used
       to solve an instance $(\exists \vb{a}) (f (\vb{a}) = g(\vb{a}))$
       of $\textsc{Csat} (\ab{A'})$ by solving $2$-$\textsc{SCsat}$ on the input
       $(\exists \vb{a}) (f (\vb{a}) = g(\vb{a}) \,\, \& \, f (\vb{a}) = g(\vb{a}))$.
       Thus the assumptions imply that for no homomorphic image $\ab{A}'$ of $\ab{A}$,
       the problem $\textsc{Csat} (\ab {A}')$ is $\NP$-complete.
       Now by \cite[Theorem~9.1]{IK:SIMC},  $\ab{A}$ is isomorphic to
       $\ab{N} \times \ab{D}$, where $\ab{N}$ is nilpotent and $\ab{D}$
       is a subdirect product of $2$-element algebras each of which is
       polynomially equivalent to a two element lattice.
       If $|\ab{D}| > 1$, then there is a homomorphic image
       $\ab{A}_2$ of $\ab{A}$ such that $\ab{A}_2$ is polynomially equivalent to
       a two element lattice. By \cite{GK:TCOP}, $2$-$\textsc{SCsat} (\ab{A}_2)$
       is $\NP$-complete, contradicting the assumptions.
       Hence $|\ab{D}| = 1$, and therefore $\ab{A}$ is nilpotent.  \end{proof}

       \section*{Acknowledgements}
       The author thanks M.\ Kompatscher for dicussions on solving equations over
      nilpotent algebras.
       These discussions took place during a workshop organized by P.\ Aglian\`o at the
      University of Siena in June 2018.
      The author also thanks A.\ F\"oldv{\'a}ri, C.\ Szab{\'o}, M.\ Kompatscher, and S.\ Kreinecker
      for their comments on preliminary versions of the manuscript.

\def\cprime{$'$}

\end{document}